\documentclass[11pt,leqno]{amsart}
\usepackage{hyperref}
\usepackage{epsfig}
\usepackage{enumitem}
\usepackage{amssymb}
\usepackage{amscd}
\usepackage[all,cmtip,matrix,arrow]{xy}
\usepackage{mathrsfs}
\usepackage{graphicx}
\usepackage{mathabx}
\usepackage{array}
\usepackage{longtable}
\xyoption{rotate}
\xyoption{pdf}
\usepackage{xparse}

\newcommand{\Implies}[2]{$\text{\ref{#1}}\implies\text{\ref{#2}}$}
\newcommand{\Iff}[2]{$\text{\ref{#1}}\iff\text{\ref{#2}}$}

\makeatletter
\DeclareRobustCommand*{\bfseries}{%
  \not@math@alphabet\bfseries\mathbf
  \fontseries\bfdefault\selectfont
  \boldmath
}
\makeatother

\newtheorem{theo}{Theorem}[section]
\newtheorem{lemma}[theo]{Lemma}
\newtheorem{defi}[theo]{Definition}
\newtheorem{prop}[theo]{Proposition}

\newtheorem{cor}[theo]{Corollary}
\newtheorem{remark}[theo]{Remark}

\newtheorem{ques}[theo]{Question}
\numberwithin{equation}{section}

\mathchardef\mhyphen="2D

\def\N{\mathbb{N}}

\def\C{\mathbb{C}}
\def\Z{\mathbb{Z}}
\def\Q{\mathbb{Q}}

\def\FF{{\mathbb F}}

\def\pre-tr{\operatorname{pre-tr}}

\def\End{\operatorname{End}}

\def\rad{\operatorname{rad}}

\newcommand{\bbar}{\overline}

\newcommand{\xto}{\xrightarrow}

\newcommand{\hto}{\hookrightarrow}
\newcommand{\onto}{\twoheadrightarrow}

\newcommand{\can}{\operatorname{can}}

\newcommand{\Gal}{\operatorname{Gal}}

\newcommand{\mk}{\mathrm k}

\newcommand{\cO}{{\mathcal O}}

\newcommand{\veps}{\varepsilon}

\newcommand{\cl}{\operatorname{cl}}

\newcommand{\ab}{\operatorname{ab}}

\newcommand{\supp}{\operatorname{Supp}}

\newcommand{\cha}{\operatorname{char}}

\newcommand{\Tor}{\operatorname{Tor}}

\newcommand{\Nm}{\operatorname{Nm}}

\newcommand{\Aut}{\operatorname{Aut}}

\newcommand{\Ind}{\operatorname{Ind}}

\newcommand{\Spf}{\operatorname{Spf}}

\newcommand{\sep}{\operatorname{sep}}

\usepackage{epsf}
\usepackage{amscd}

\newcommand{\m}{\mathfrak{m}}

\title[On the Hahn-Witt series and their generalizations]
{On the Hahn-Witt series and their generalizations}

\author{Alexander I. Efimov}
\address{The Hebrew University of Jerusalem}
\email{efimov@mccme.ru}
\thanks{The author was partially supported by the European Research Council (ERC, CurveArithmetic, 101078157).}

\begin{document}

\begin{abstract} In this paper we study the field of Hahn-Witt series $HW(\bbar{\FF}_p)$ with residue field $\bbar{\FF}_p$ (also known as a $p$-adic   Malcev-Neumann field \cite{La86, P93}), and its generalizations. Informally, the Hahn-Witt series are possibly infinite linear combinations of rational powers of $p,$ in which the coefficients are Teichm\"uller representatives, and the set of exponents is well-ordered. They form an algebraically closed extension of $\Q_p,$ with a canonical automorphism $\varphi,$ coming from the absolute Frobenius of $\bbar{\FF}_p.$ We prove that the action of $\varphi$ on the $p$-power roots of unity is given by $\varphi(\zeta)=\zeta^{-1},$ answering a question of Kontsevich.

More generally, we consider the $\pi$-typical Hahn-Witt series $HW_{(K,\pi)}(\bbar{\FF}_q)$, where $\pi$ is a uniformizer in a local field $K$ with residue field $\FF_q.$ Again, this field is an algebraically closed extension of $K,$ and it has a canonical automorphism $\varphi_{\pi},$ coming from the relative Frobenius of $\bbar{\FF}_q$ over $\FF_q.$ We prove that the action of $\varphi_{\pi}$ on the maximal abelian extension $K^{\ab}$ corresponds via local class field theory to the uniformizer $-\pi\in K^*.$ 
\end{abstract}


\maketitle

\tableofcontents

\section{Introduction}

In this paper we study the field $HW(\bbar{\FF}_p)$ of Hahn-Witt series and its generalizations. The field $HW(\bbar{\FF}_p)$ is also known as a $p$-adic Malcev-Neumann field \cite{La86, P93}, with residue field $\bar{\FF}_p$ and value group $\Q.$ Informally, the elements of this field are expressions of the form $\sum\limits_{i\in \Q}[a_i]p^i,$ where $a_i\in\bbar{\FF}_p$ and the set $\{i: a_i\ne 0\}\subset\Q$ is well-ordered (i.e. there are no infinite strictly decreasing sequences of elements in this set). Here $[a_i]$ is the Teichm\"uller representative of $a_i.$ The addition and multiplication are defined essentially by the same rules as for the usual Witt vectors of perfect $\FF_p$-algebras. Replacing $\bbar{\FF}_p$ with an arbitrary perfect $\FF_p$-algebra $R,$ we obtain a functor
$$HW:\{\text{perfect }\FF_p\text{-algebras}\}\to \{\Q_p\text{-algebras}\}.$$
We refer to Section \ref{sec:Hahn_series} and Definition \ref{def:Hahn_Witt} (applied to $K=\Q_p,$ $\pi=p$) for details.

By \cite[Theorem 1]{P93}, the field $HW(\bbar{\FF}_p)$ is maximally complete, i.e. it has no proper immediate extensions (that is, proper extensions with a valuation extended from $HW(\bbar{\FF}_p),$ with the same value group and with the same residue field). Note that the Gelfand-Mazur theorem can be thought of as a statement about ``maximal completeness'' of the field of complex numbers: there are no proper extensions of $\C$ with a norm extended from $\C.$

The maximal completeness of the field $HW(\bbar{\FF}_p)$ implies that it is algebraically closed \cite[Proposition 6]{P93}. The $p$-adic completion of the integral closure of $W(\bbar{\FF}_p)$ in $HW(\bbar{\FF}_p)$ was described by Kedlaya \cite{Ke01} (in fact, in loc. cit. $\bbar{\FF}_p$ is replaced by an arbitrary algebraically closed field of characteristic $p,$ but for $\bbar{\FF}_p$ the description simplifies considerably). The description in loc. cit. implies in particular that if a Hahn-Witt series $f=\sum\limits_i [a_i]p^i$ is algebraic over $\Q_p,$ then the set $\{i: a_i\ne 0\},$ considered as an ordinal (the order type of $f$), is at most $\omega^{\omega}.$ This is a very strong condition: for a general Hahn-Witt series, its order type can be arbitrary at most countable ordinal.

Another remarkable property of Hahn-Witt series which are algebraic over $\Q_p$ was proved by Lampert \cite{La86}: all the accumulation values of the set of exponents are rational. Namely, by \cite[Theorem 2]{La86} the Hahn-Witt series with this property form an algebraically closed field. For more results on the algebraic closure of $\Q_p$ in $HW(\bbar{\FF}_p)$ we refer to \cite[Section 7]{P93}.

Unlike the Tate field $\C_p$ (which is defined only up to a non-canonical continuous isomorphism), the field $HW(\bbar{\FF}_p)$ has the canonical automorphism $\varphi,$ which comes from the (absolute) Frobenius automorphism of $\bbar{\FF}_p.$ Explicitly, we have $$\varphi(\sum\limits_{i\in \Q}[a_i]p^i)=\sum\limits_{i\in \Q}[a_i^p]p^i.$$
This automorphism can be thought of as a $p$-adic analogue of complex conjugation. Note that $\varphi$ induces a canonical conjugacy class in the absolute Galois group $G_{\Q_p},$ and in particular in $G_{\Q}.$

The subfield of $\varphi$-invariants in $HW(\bbar{\FF}_p)$ is exactly $HW(\FF_p).$ In particular, the field  $HW(\FF_p)$ is quasi-finite (i.e. its Galois group is isomorphic to $\hat{\Z}$), and its finite extensions are exactly of the form $HW(\FF_{p^n}),$ $n\geq 1.$ 

Now, consider the group $\mu_{p^{\infty}}$ of $p$-power roots of unity in $HW(\bbar{\FF}_p).$ This group is non-canonically isomorphic to $\Q_p/\Z_p,$ hence there is a unique invertible $p$-adic integer $u_p\in\Z_p^*$ such that $$\varphi(\zeta)=\zeta^{u_p},\quad \zeta\in\mu_{p^{\infty}}.$$

The following natural question, attributed to Kontsevich, was suggested to me by Robert Burklund and Achim Krause.

\begin{ques}[Kontsevich]\label{ques:u_p} What is $u_p\in\Z_p^*$?\end{ques}

In this paper we obtain the following answer as a special case of our main result (Theorem \ref{th:main}).

\begin{theo}\label{th:main_for_Hahn_Witt} Within the above notation, we have $u_p=-1$ for all prime numbers $p.$\end{theo}

In Section \ref{sec:computations} we do some computations of (approximations of) $p$-power roots of unity in $HW(\bbar{\FF}_p).$ These computations illustrate that a direct approach to Question \ref{ques:u_p} is unrealistic.

We now formulate a generalization of Theorem \ref{th:main_for_Hahn_Witt}, which will be also necessary for the proof. Let $K$ be a local field, i.e. a complete discrete valuation field with a finite residue field. Denote by $\cO_K$ the valuation ring, and choose a uniformizer $\pi\in\cO_K.$ We identify the residue field $\mk=\cO_K/\pi$ with $\FF_q.$ Recall that we have a functor of $\pi$-typical Witt vectors
$$W_{(K,\pi)}:\{\cO_K\text{-algebras}\}\to \{\cO_K\text{-algebras}\},$$
introduced in \cite{Dr76}. Our notation is somewhat cumbersome, but we want to stress that we make a choice of the uniformizer $\pi$ in what follows.

The functor $W_{(K,\pi)}$ is defined similarly to the usual functor of $p$-typical Witt vectors, but the ghost coordinates are now given by
$$w_n(x_0,x_1,\dots)=\sum\limits_{i=0}^n \pi^i x_i^{q^{n-i}},\quad n\geq 0.$$
Also, the Joyal's approach to $p$-typical Witt vectors via $\delta$-rings \cite{J85} has a natural generalization to the case of $\pi$-typical Witt vectors, see \cite{Bo11}. Here $\delta$-rings are replaced by $\delta_{\pi}$-$\cO_E$-algebras. 

There is a unique functorial multiplicative section $[\cdot]:R\to W_{(K,\pi)}(R)$ of the projection $w_0:W_{(K,\pi)}(R)\to R,$ such that $w_n([r])=r^{q^n}$ for $n>0,$ $r\in R.$ As usual, $[r]$ is called the Teichm\"uller representative of $r.$ 

If $R$ is a perfect $\FF_q$-algebra, considered as an $\cO_K$-algebra, then $W_{(K,\pi)}(R)$ is a $\pi$-complete, $\pi$-torsion free $\cO_K$-algebra. The functor
$$W_{(K,\pi)}:\{\text{perfect }\FF_q\text{-algebras}\}\to \{\pi\text{-complete, }\pi\text{-torsion free }\cO_K\text{-algebras}\}$$
is the partially defined left adjoint functor to the mod $\pi$ reduction
$$\{\pi\text{-complete, }\pi\text{-torsion free }\cO_K\text{-algebras}\}\to \{\FF_q\text{-algebras}\}.$$
In particular, we have $W_{(K,\pi)}(\FF_q)=\cO_K.$ If $R$ is a perfect $\FF_q$-algebra, then any element of $W_{(K,\pi)}(R)$ can be uniquely written as an infinite sum $\sum\limits_{i=0}^{\infty}[a_i]\pi^i,$ where $a_i\in R.$ 

Now, the $K$-algebra of $\pi$-typical Hahn-Witt series $HW_{(K,\pi)}(R)$ of a perfect $\FF_q$-algebra $R$ can be defined informally as the set of expressions of the
form $\sum\limits_{i\in\Q} [a_i]\pi^i,$ where $a_i\in R$ and the set $\{i: a_i\ne 0\}$ is well-ordered. The addition and multiplication are defined by the same rules as for the usual $\pi$-typical Witt vectors, and the map $K\to HW_{(K,\pi)}(R)$ is the tautological inclusion (assuming that $R\ne 0$). We refer to Section \ref{sec:Hahn_series} and Definition \ref{def:Hahn_Witt} for details. We obtain a functor
$$HW_{(K,\pi)}:\{\text{perfect }\FF_q\text{-algebras}\}\to \{K\text{-algebras}\}.$$

Now we formulate a generalization of Theorem \ref{th:main_for_Hahn_Witt}. 
Consider the $K$-algebra $HW_{(K,\pi)}(\bbar{\FF}_q).$ This $K$-algebra is in fact an algebraically closed field, by a straightforward generalization of \cite[Theorem 1, Proposition 6]{P93}. By functoriality, the relative Frobenius automorphism of $\bbar{\FF}_q$ over $\FF_q$ induces a $K$-linear automorphism $\varphi_{\pi}$ of $HW_{(K,\pi)}(\bbar{\FF}_q).$ Explicitly, we have  
$$\varphi_{\pi}(\sum\limits_i [a_i] \pi^i)=\sum\limits_i [a_i^q] \pi^i.$$

Denoting by $K^{\sep}$ the separable algebraic closure of $K$ in $HW_{(K,\pi)}(\bbar{\FF}_q),$ we see that $\varphi_{\pi}$ induces an element of the absolute Galois group $\Gal(K^{\sep}/K),$ which induces the relative Frobenius automorphism on the residue field. In particular, this element is contained in the Weil group $W_K,$ and it is of degree $1.$  Our main result is the following.

\begin{theo}\label{th:main} Let $K$ be a local field, and choose a uniformizer $\pi\in K.$ We identify the residue field $\cO_K/\pi$ with $\FF_q.$ Consider the canonical automorphism $\varphi_{\pi}$ of the field of $\pi$-typical Hahn-Witt series $HW_{(K,\pi)}(\bbar{\FF}_q).$ Consider the induced element $\bbar{\varphi}_{\pi}\in W_K^{\ab}$ of degree $1.$ Applying the inverse of the Artin reciprocity map $\theta_K^{-1}:W_K^{\ab}\xto{\sim} K^*$ to $\bbar{\varphi}_{\pi},$ we obtain a uniformizer $\can_K(\pi)\in K^*.$ Then we have $$\can_K(\pi)=-\pi.$$\end{theo}

Note that Theorem \ref{th:main_for_Hahn_Witt} is indeed a special case of Theorem \ref{th:main}: the equality $\can_{\Q_p}(p)=-p$ exactly means that $\varphi$ is acting on the $p$-power roots of unity by $\varphi(\zeta)=\zeta^{-1}.$ 

We have an immediate corollary, which is quite surprising (to the author).

\begin{cor}Let $K$ be a local field, and denote by $W_K^{\deg=1}$ the set of elements of the Weil group of degree $1,$ and similarly for $W_K^{\ab,\deg =1}.$ Denote by $\cl(W_K^{\deg=1})$ the set of conjugacy classes. Then the surjection $\cl(W_K^{\deg=1})\onto W_K^{\ab,\deg=1}$ has a canonical section which sends $\theta_K(\pi)$ to the conjugacy class coming from $\varphi_{-\pi}.$\end{cor}

It would be interesting to understand more about the conjugacy classes in $W_K^{\deg=1}$ coming from $\varphi_{\pi}.$

\begin{remark} Using Lubin-Tate theory \cite{LT65} and Hazewinkel's classification of the logarithms of Lubin-Tate formal group laws \cite[Proposition I.8.3.6]{Haz78}, we observe that Theorem \ref{th:main} is equivalent to the following: the zeroes of the function $$H(x)=x-\frac{x^q}{\pi}+\frac{x^{q^2}}{\pi^2}-\frac{x^{q^{3}}}{\pi^3}+\dots$$
in (the maximal ideal of the valuation ring of) $HW_{(K,\pi)}(\bar{\FF}_q)$ are fixed by $\varphi_{\pi}.$ In other words, these zeroes are contained $HW_{(K,\pi)}(\FF_q).$ Again, the computational approach to this statement is unrealistic, unless $\cha K=p.$ 
\end{remark}

We will prove Theorem \ref{th:main} in Section \ref{sec:proof_of_main}, using essentially only the standard results from local class field theory, Lubin-Tate theory and local ramification theory, together with the Krasner's lemma.

{\noindent{\bf Notation.}} {\it All the valuations below will take values in $\Q\cup\{\infty\},$ and they are not normalized unless otherwise stated. The valuations will be denoted by $v.$ In particular, for a discrete valuation field $K$ with a uniformizer $\pi,$ in general we do not require that $v(\pi)=1$ unless we choose to do so.}

{\noindent{\bf Acknowledgements.}} I am grateful to Robert Burklund and Achim Krause for introducing me to Hahn-Witt series and for suggesting the Question \ref{ques:u_p}.

\section{Hahn series and $\pi$-typical Hahn-Witt series}
\label{sec:Hahn_series}
 
In this section we first give some background on Hahn series \cite{Hahn95, Kr32, Pa77}, then we define the $\pi$-typical Hahn-Witt series (Definition \ref{def:Hahn_Witt}). We also summarize the properties of $\pi$-typical Hahn-Witt series of a perfect field (Theorem \ref{th:Poonen}). Almost all the statements in this section, including Theorem \ref{th:Poonen}, are obtained by a straightforward generalization of results from \cite{P93}. 
   
For a commutative ring $R$ we have a ring of Hahn series $R((t^{\Q})).$ It is formed by expressions $f=\sum\limits_{i\in\Q} a_i t^i,$ where $a_i\in R$ and the support $$\supp(f):=\{i\in\Q: a_i\ne 0\}$$ is well-ordered. The sum and product of such expressions are well-defined. Namely, by \cite[Lemma 2.9]{Pa77} for well-ordered subsets $A,B\subset\Q,$ the sets $A\cup B$ and $A+B$ are also well-ordered, and the map
$$A\times B\to A+B,\quad (a,b)\mapsto a+b,$$
has finite fibers. Moreover, the subring $R[[t^{\Q}]]\subset R((t^{\Q}))$ (formed by expressions with non-negative exponents) is $t$-adically complete. Just like in the case of usual power series, the Jacobson radical $\rad(R[[t^{\Q}]])$ consists of Hahn series whose constant term is contained in $\rad(R).$ 

Suppose that $r\in R$ is a non-zero-divisor such that $R$ is $r$-adically complete. Following the terminology of \cite[Section 4]{P93} (in the case $r=p$), we say that an element $f=\sum\limits_i a_i t^i\in R((t^{\Q}))$ is a null series (with respect to $r$) if for any $i\in\Q$ we have
$$\sum\limits_{n\in\Z} a_{i+n}r^n=0\in R[r^{-1}].$$ The same proof as in \cite[Proposition 3]{P93} shows that the null series form an ideal in $R((t^{\Q})).$ 
Moreover, we observe that this ideal is principal, and it is generated by the non-zero-divisor $(t-r).$ Namely, if $f=\sum\limits_i a_i t^i$ is a null-series, then we have
$$\frac{f}{t-r}=\sum\limits_{i\in\Q}\left(\sum\limits_{n\in\Z_{\geq 0}}a_{i+n+1}r^n\right)t^i.$$ Note that $\frac{f}{t-r}$ is a well-defined element of 
$R((t^{\Q})),$ since $\supp(\frac{f}{t-r})\subset\supp(f)+\N,$ and the latter set is well-ordered. 

Choose a set-theoretic section $s:R/r\to R,$ such that $s(0)=0.$ Then the same proof as in \cite[Proposition 4]{P93} shows that any element of $R((t^{\Q}))/(t-r)$ has a unique representative of the form
\begin{equation}\label{eq:representative_of_series} \sum\limits_{i\in\Q}s(b_i) t^i\in R((t^{\Q})),\quad b_i\in R/x,\quad \{i\in\Q: b_i\ne 0\}\text{ is well-ordered}.\end{equation}

\begin{defi}
We denote by $R((r^{\Q}))$ the quotient $R((t))^{\Q}/(t-r),$ assuming that $r\in R$ is a non-zero-divisor such that $R$ is $r$-adically complete.
We will write $\sum\limits_i s(b_i)r^i$ for the image of the element \eqref{eq:representative_of_series}. 
\end{defi} 

Note that the assignment $(R,r)\mapsto R((r^{\Q}))$ is functorial, and we have a natural inclusion $R[r^{-1}]\hto R((r^{\Q})).$ The following (almost tautological) observation will be very useful.

\begin{prop}\label{prop:extracting_a_root} Let $R$ be a commutative ring and let $r\in R$ be a non-zero-divisor such that $R$ is $r$-adically complete. Choose a section $s:R/r\to R,$ such that $s(0)=0.$  Denote by $R'=R[r^{\frac1n}]$ the ring $R[X]/(X^n-r),$ where $n>0$ is an integer. We identify $R/r\cong R'/r^{\frac1n},$ and we choose the map $R/r\xto{s} R\to R'$ as a section. Then we have a natural isomorphism $$R((r^{\Q}))\xto{\sim} R'(((r^{\frac1n})^{\Q})),\quad \sum\limits_i s(b_i)r^i\mapsto \sum\limits_i s(b_i)(r^{\frac1n})^{ni}.$$\end{prop}

\begin{proof} Indeed, the inverse isomorphism is given by
\begin{multline*} R'(((r^{\frac1n})^{\Q}))= R'((t^\Q))/(t-r^{\frac1n})\cong R((t^{\Q}))[X]/(X^n-r,t-X)\\ \cong R((t^{\Q}))/(t^n-r)=R((r^{\Q})).\qedhere\end{multline*}\end{proof}

We now define the $\pi$-typical Hahn-Witt series. Recall that we denote by $W_{(K,\pi)}(-)$ the $\pi$-typical Witt vectors, as defined in \cite{Dr76}.

\begin{defi}\label{def:Hahn_Witt} Let $K$ be a local field with a chosen uniformizer $\pi,$ and we identify $\cO_K/\pi$ with $\FF_q.$ For a perfect $\FF_q$-algebra $R,$ we define the $K$-algebra of Hahn-Witt series of $R$ as follows:
$$HW_{(K,\pi)}(R):=W_{(K,\pi)}(R)((\pi^{\Q}))=W_{(K,\pi)}(R)((t^{\Q}))/(t-\pi).$$\end{defi}

Choosing the section $[\cdot]:R\cong W_{(K,\pi)}(R)/\pi\to W_{(K,\pi)}(R)$ to be the unique multiplicative section (given by the Teichm\"uller representatives), we see that the elements of $HW_{(K,\pi)}(R)$ are expressions of the form 
$$\sum\limits_{i\in\Q}[a_i] \pi^i,\quad a_i\in R,\quad \{i\in\Q: a_i\ne 0\}\text{ is well-ordered}.$$
 
A straightforward generalization of the results of \cite[Sections 4 and 5]{P93} gives the following.

\begin{theo}\label{th:Poonen} \cite{P93} Let $K$ and $\pi$ be as above, and let $E$ be a perfect field, which is an extension of $\FF_q.$

1) The $K$-algebra $HW_{(K,\pi)}(E)$ is a field.

2) The valuation $K\to\Z\cup\{\infty\},$ sending $\pi$ to $1,$ extends to a valuation $v:HW_{(K,\pi)}(E)\to\Q\cup\{\infty\},$ given by $$v(\sum\limits_i [a_i] \pi^i)=\inf\{i\in\Q: a_i\ne 0\},$$ where $\inf(\emptyset)=\infty.$ The inclusion $K\subset HW_{(K,\pi)}(E)$ induces an isomorphism on the residue fields. 

3) $HW_{(K,\pi)}(E)$ is a maximally complete valued field, i.e. it has no immediate extensions other than itself. Recall that if $F$ is a field with a valuation $v:F\to\Q\cup \{\infty\},$ then an immediate extension of $F$ is a field extension $F'$ of $F$ with a valuation $v':F'\to\Q\cup\{\infty\}$ such that $v'_{\mid F}=v,$ and $F$ and $F'$ have the same value group and the same residue field.

4) If $E$ is algebraically closed, then $HW_{(K,\pi)}(E)$ is also algebraically closed.\end{theo} 

\begin{proof} We give references to the corresponding results from \cite{P93}. The proofs of 1)-3) are obtained from the proofs from \cite{P93} by replacing $p$ with $\pi,$ and restricting to the case $G=\Q$ (of course, all the statements are still valid for a general totally ordered value group $G$ containing $\Z,$ which must be divisible in part 4)). Part 4) follows from 3) by \cite[Proposition 6]{P93}.

Part 1) is a generalization of \cite[Corollary 3]{P93}. Part 2) is a generalization of \cite[Proposition 5]{P93}. Part 3) is a generalization of \cite[Theorem 1]{P93}.\end{proof}

As mentioned in the introduction, we will be interested in the canonical automorphism 
$$\varphi_{\pi}:HW_{(K,\pi)}(\bbar{\FF}_q)\to HW_{(K,\pi)}(\bbar{\FF}_q), \quad \varphi_{\pi}(\sum\limits_i [a_i]\pi^i)=\sum\limits_i [a_i^q]\pi^i,$$
especially in its action on the maximal abelian extension $K^{\ab}$ of $K.$

If $K=\Q_p$ and $\pi=p,$ then we write simply $HW(-)$ and $\varphi$ instead of $HW_{(K,\pi)}(-)$ and $\varphi_{\pi},$ assuming the choice of a prime $p.$ 

\begin{remark}If $K$ is a finite extension of $\Q_p,$ then the fields $HW_{(K,\pi)}(\bbar{\FF}_q)$ and $HW(\bbar{\FF}_p)$ are non-canonically isomorphic as valued fields by \cite[Corollary 6]{P93}. Namely, there is only one maximally complete field of characteristic zero with value group $\Q$ and residue field $\bbar{\FF}_p$ (up to a non-canonical isomorphism). Moreover, the group of continuous automorphisms of $HW(\bbar{\FF}_p)$ is huge, even over $\C_p,$ see \cite[Proposition 11]{P93}.\end{remark}

\section{Some computations}
\label{sec:computations}

This section, which can be skipped on the first reading, contains some computations of the $p$-power roots of unity in $HW(\bbar{\FF}_p).$ These are mostly to demonstrate that it is probably impossible to prove Theorem \ref{th:main_for_Hahn_Witt} by a direct approach. Some of the computations below were done previously by Lampert \cite{La86}, however the formulas in \cite[Theorem 1]{La86} are incorrect (see below). Also, some of the computations for $p=2$ below were done recently by Achim Krause (private communication). Recall  from the introduction that $u_p$ denotes the unique $p$-adic integer such that $\varphi(\zeta)=\zeta^{u_p},$ where $\zeta$ runs through $p$-power roots of unity in $HW(\bbar{\FF}_p).$

Looking at the Newton polygons, we see that to prove that some $p$-power root of unity is contained in $HW(\FF_{p^n})$ for a particular $n,$ it suffices to check this for the coefficients of $p^i$ for $i\leq\frac{1}{p-1}.$ Indeed, for higher degrees the coefficients are uniquely determined by successive approximation, and they will automatically be (Teichm\"uller representatives of elements) in $\FF_{p^n}.$ We refer to \cite{La86} for details. 

Now, assuming $p>2,$ we first compute $$\sqrt[p]{1}=\sum\limits_{i=0}^{p-1}\left[\frac{a^i}{i!}\right]\cdot p^{\frac{i}{p-1}}+\left[a\cdot\frac{1+(p-1)!}{p}\right]\cdot p^{\frac{p}{p-1}}+(\text{higher order terms}),$$
where $a\in \FF_{p^2},$ $a^{p-1}=-1.$ This is simply the computation within the finite extension $W(\FF_{p^2})[p^{\pm \frac{1}{p-1}}]$ of $\Q_p.$
In particular, we have $\sqrt[p]{1}\in HW(\FF_{p^2}).$ Since $a^p=-a,$ we have $u_p\equiv -1\text{ mod }p.$ We remark that the formula for $\sqrt[p]{1}\text{ mod }p$ in \cite[Theorem 1]{La86} is incorrect for $p\geq 5:$ this can be checked by raising it to the power $p$ and looking at the coefficient of $p^{\frac{p+2}{p-1}}$ (which equals $[\frac{\zeta_1^{3}}{12}]$ in the notation of loc. cit.). 

Knowing the expression for $\sqrt[p]{1}$ up to degree $p^{\frac{p}{p-1}}$ allows to compute the expression for (one of the primitive roots) $\sqrt[p^2]{1}$ up to degree $p^{\frac{1}{p-1}}.$ We obtain:
$$\sqrt[p^2]{1}=\sum\limits_{i=0}^{p-1}\left[\frac{(-a)^i}{i!}\right]\cdot p^{\frac{i}{p(p-1)}}+\sum\limits_{n=2}^{\infty} [-a]\cdot p^{\frac{1}{p-1}-\frac{1}{p^n}}+0\cdot p^{\frac{1}{p-1}}+(\text{higher order terms}).$$
In particular, $\sqrt[p^2]{1}\in HW(\FF_{p^2})$ and $u_p\equiv -1\text{ mod }p^2.$ However, computing $\sqrt[p^3]{1}$ up to degree $p^{\frac{1}{p-1}}$ seems to be just too difficult, and for general $p$-power roots of unity this seems to be simply impossible. Assuming Theorem \ref{th:main_for_Hahn_Witt}, we can only say that $\sqrt[p^n]{1}\in HW(\FF_{p^2})$ for all $n.$  

For $p=2$ the computational approach gives a slightly better result, i.e. $u_2\equiv -1\text{ mod }16,$ or equivalently $\sqrt{-1}\not\in HW(\FF_2)$ and $\sqrt[16]{-1}\in HW(\FF_4).$ First, Lampert \cite{La86} computed: $$\sqrt{-1}=(2^{\frac12}+2^{\frac34}+2^{\frac78}\dots)+[a]\cdot 2^1+(\text{higher order terms}),$$
where $a\in\FF_4,$ $a^2+a=1.$ We see that $\sqrt{-1}\not\in HW(\FF_2)$ and $u_2\equiv -1\text{ mod }4.$ We also see that for any non-zero element $x\in HW(\FF_2),$ either $x$ or $-x$ is a square, but not both. This follows from Kummer theory, since the Galois group of $HW(\FF_2)$ is isomorphic to $\hat{\Z}.$


Next, we can directly compute that $\sqrt{1+2^{\frac12}}$ is contained in $HW(\FF_2):$
$$\sqrt{1+2^{\frac12}}=1+\sum\limits_{n=2}^{\infty}2^{1-\frac{3}{2^n}}+\sum\limits_{n=4}^{\infty}2^{1-\frac{1}{2^n}}+0\cdot 2^1+(\text{higher order terms}).$$
Now, either $2+\sqrt{2+\sqrt{2}}$ or $-2+\sqrt{2+\sqrt{2}}$ is a square in $HW(\FF_2),$ but not both of them, since their product equals $-\sqrt{2}(1+\sqrt{2})^{-1},$ which is not a square. We obtain
$$\sqrt[16]{-1}=\frac{1}{2}\sqrt{\pm 2+\sqrt{2+\sqrt{2}}}+\frac{1}{2}\sqrt{\pm 2-\sqrt{2+\sqrt{2}}}\cdot\sqrt{-1}.$$
In particular, $u_2\equiv -1\text{ mod }16.$ Moreover, $u_2\equiv -1\text{ mod }32$ if and only if $2+\sqrt{2+\sqrt{2}}$ is a square in $HW(\FF_2).$ This is the case by Theorem \ref{th:main_for_Hahn_Witt}, but computing the explicit Hahn-Witt series expression for $\sqrt{2+\sqrt{2+\sqrt{2}}}$ seems to be too difficult, even if we try do it only up to the relevant degree $2^{\frac98}$.

\begin{remark}\label{remark:p=2_via_square_roots} In fact, Theorem \ref{th:main_for_Hahn_Witt} for $p=2$ is equivalent to the following: all the square roots
$$\sqrt{2+\sqrt{2+\dots\sqrt{2}}}$$
are contained in $HW(\FF_2).$\end{remark}

\section{Local class field theory}
\label{sec:CFT}

In this section we will recall only the main statements from the local class field theory, that is, the properties of the local Artin reciprocity maps. We refer to \cite{CF67, N99, M20} for a detailed study. 

We will deduce a very useful relation between the uniformizers $\can_K(\pi),$ introduced in Theorem \ref{th:main}, for different pairs $(K,\pi).$ The precise statement is Proposition \ref{prop:can_related_via_Norm}.

 Let $K$ be a local field, and denote by $\mk$ its (finite) residue field. Denote by $v:K^*\to\Z$ the surjective discrete valuation. Denote by $G_K$ the absolute Galois group of $K,$ and denote by $G_K^{\ab}$ (the profinite completion of) its abelianization. Then $G_K^{\ab}$ is the Galois group of the maximal abelian extension $K^{\ab}$ of $K.$ 

It is convenient to replace the group $G_K$ with the Weil group $W_K\subset G_K.$ By definition, $W_K$ is the preimage of $\Z\subset \hat{\Z}$ under the natural map $G_K\to G_{\mk}\cong\hat{\Z},$ and similarly for $W_K^{\ab}\subset G_K^{\ab}.$ The induced map $\deg:W_K\to\Z$ (which of course factors through $W_K^{\ab}$) is called the degree map. 

The main statement of local class field theory is the existence of a natural isomorphism, called Artin reciprocity map:
$$\theta_K:K^*\xto{\sim}W_K^{\ab}.$$
The degree map $\deg:W_K^{\ab}\to\Z$ corresponds to the valuation:
$$
\begin{CD}
K^* @>v>> \Z\\
@V\theta_KVV @|\\
W_K^{\ab} @>\deg>> \Z
\end{CD}
$$
Given a finite extension $L/K,$ the natural map $W_L^{\ab}\to W_K^{\ab}$ corresponds to the norm map:
\begin{equation}
\label{eq:comm_square_norm}
\begin{CD}
L^* @>\Nm_{L/K}>> K^*\\
@V\theta_L VV @V\theta_K VV\\
W_L^{\ab} @>>> W_K^{\ab}.
\end{CD}
\end{equation}
The inclusion $K^*\to L^*$ corresponds to the transfer map:
$$
\begin{CD}
K^* @>>> L^*\\
@V\theta_K VV @V\theta_LVV\\
W_K^{\ab} @>T>> W_L^{\ab}.
\end{CD}
$$
Recall that for an inclusion of (say, abstract) groups $H\subset G$ of finite index, the transfer map $T:G^{\ab}\to H^{\ab}$ is given by the composition
$$T:G^{\ab}\cong \Tor_1^{\Z[G]}(\Z,\Z)\to\Tor_1^{\Z[G]}(\Z,\Ind_H^G(\Z))\cong \Tor_1^{\Z[H]}(\Z,\Z)\cong H^{\ab}.$$

Now, let $K$ be a local field with a uniformizer $\pi,$ and we identify $\cO_K/\pi\cong\FF_q.$ Recall that the uniformizer $\can_K(\pi)\in K,$ introduced in Theorem \ref{th:main}, is characterized by the following: the action of $\theta_K(\can_K(\pi))$ on $K^{\ab}\subset HW_{(K,\pi)}(\bbar{\FF}_q)$ is the same as the action of $\varphi_{\pi}\in\Aut_K(HW_{(K,\pi)}(\bbar{\FF}_q)).$

\begin{prop}\label{prop:can_related_via_Norm} Consider the extension $K'=K(\pi^{\frac1n})$ of $K,$ for some $n>0.$ Then we have \begin{equation}\label{eq:can_related_via_Norm} \Nm_{K'/K}(\can_{K'}(\pi^{\frac1n}))=\can_K(\pi).\end{equation}\end{prop}

\begin{proof}
By Proposition \ref{prop:extracting_a_root}, we have a natural isomorphism \begin{equation}\label{eq:iso_for_obvious_extensions} \psi:HW_{(K,\pi)}(\bbar{\FF}_q)\xto{\sim}HW_{(K',\pi^{\frac1n})}(\bbar{\FF}_q),\quad \psi(\sum\limits_i [a_i]\pi^i)=\sum\limits_i [a_i](\pi^{\frac1n})^{ni},\end{equation}
where $a_i\in\bbar{\FF}_q.$
Moreover, $\psi$ is compatible with the canonical automorphisms $\varphi_{\pi}$ and $\varphi_{\pi^{\frac1n}}.$ Hence, the equality \eqref{eq:can_related_via_Norm} follows from the commutative square \eqref{eq:comm_square_norm}, applied to the extension $K'/K.$
\end{proof}

\section{Ramification theory}
\label{sec:ramification}

We will need only one well known statement about the norm map for a particular wildly ramified extension. Let $K$ be a local field of characteristic $0$ and with the residue characteristic $p.$ Choose a uniformizer $\pi\in K.$ Consider the extension $K'=K(\pi^{\frac1p})$ of $K.$ 

\begin{lemma}\label{lem:norms_for_wildly_ramified} Let $n$ and $k$ be positive integers, such that $$n\leq k,\quad (n-\frac{k}{p})v(\pi)\leq v(p).$$ Then we have
$$\Nm_{K'/K}(1+\pi^{\frac{k}{p}}\cO_{K'})\subset 1+\pi^n\cO_K.$$\end{lemma}

\begin{proof} This is completely standard, but we were not able to find a reference (note that for $p>2$ the extension $K'/K$ does not have to be Galois). We give a proof for completeness. 

It suffices to prove that for $m\geq k$ and for $a\in\cO_K/\pi$ we have
$$\Nm_{K'/K}(1-[a]\pi^{\frac{m}{p}})\in 1+\pi^n\cO_K.$$
 If $m$ is not divisible by $p,$ then we have $$\Nm_{K'/K}(1-[a]\pi^{\frac{m}{p}})=1-[a^p]\pi^m\in 1+\pi^n\cO_K,$$
 since $m\geq n.$
If $m$ is divisible by $p,$ then we have
$$\Nm_{K'/K}(1-[a]\pi^{\frac{m}{p}})=(1-[a]\pi^{\frac{m}{p}})^p\in 1+p\pi^{\frac{m}{p}}\cO_K+\pi^m\cO_K\subset 1+\pi^n\cO_K,$$
since $m\geq n$ and $$v(p)+v(\pi^{\frac{m}{p}})\geq (n-\frac{k}{p})v(\pi)+\frac{m}{p}v(\pi)\geq v(\pi^n).$$ This proves the lemma.
\end{proof}

\section{Lubin-Tate theory}
\label{sec:LT_theory}

Let $K$ be a local field with the ring of integers $\cO_K,$ and choose a uniformizer $\pi\in K.$ We identify $\cO_K/(\pi)\cong\FF_q.$ Choose a separable algebraic closure $K^{\sep}$ of $K,$ and denote by $K^{\ab}\subset K^{\sep}$ the maximal abelian extension of $K.$ Denote by $K_{\pi}\subset K^{\ab}$ the subfield fixed by $\theta_K(\pi).$ Then $K_{\pi}$ is a maximal totally ramified abelian extension of $K,$ and moreover the assignment $\pi\mapsto K_{\pi}$ defines a bijection
\[
\xymatrix@1{
\left\{
\text{
\begin{minipage}[c]{1.15in}
uniformizers of $K$
\end{minipage}
}
\right\}
\ar[r]^-{\sim}
&
\left\{
\text{
\begin{minipage}[c]{1.8in}
maximal totally ramified abelian extensions of $K$
\end{minipage}
}
\right\}
}
\]

The composition $$\bbar{\theta}_{K,\pi}:\cO_K^*\to K^*\xto{\theta_K}W_K^{\ab}\to\Gal(K_{\pi}/K)$$
is an isomorphism.

The Lubin-Tate theory \cite{LT65} provides a geometric (more precisely, formal scheme-theoretic) interpretation of the field $K_{\pi}$ and of the isomorphism $\bbar{\theta}_{K,\pi}:\cO_K^*\xto{\sim} \Gal(K_{\pi}/K).$ We will very briefly recall the main statements, and we refer to \cite{LT65, CF67, Haz78} for details. Then we will deduce an application (Lemma \ref{lem:LT_for_approximating_can_K}) to the uniformizers $\can_K(\pi)$ introduced in Theorem \ref{th:main}. 

The main ingredient is the unique (up to a non-canonical isomorphism) $1$-dimensional formal $\cO_K$-module over $\cO_K$ such that the action of $\pi$ induces the relative Frobenius endomorphism modulo $\pi.$ More precisely, there exists a pair $(F,\alpha),$ where $F(X,Y)\in\cO_K[[X,Y]]$ is a formal group law over $\cO_K$ and $$\alpha:\cO_K\to\End(\Spf\cO_K[[X]],F),\quad \alpha(r)=[r]_F,$$ is a homomorphism of rings satisfying
$$[r]_F(X)\equiv rX\text{ mod }X^2\text{ for }r\in\cO_K,\quad [\pi]_F(X)\equiv X^q\text{ mod }\pi.$$
Such $F$ is called a Lubin-Tate formal group law, and the formal $\cO_K$-module $(F,\alpha)$ is called a Lubin-Tate formal $\cO_K$-module.

Such a pair $(F,\alpha)$ is unique up to an isomorphism, i.e. up to a formal change of the coordinate. Moreover, the pair $(F,\alpha)$ is uniquely determined  by the power series $f(X)=[\pi]_F(X),$ which must satisfy the condition
\begin{equation}\label{eq:LT_action_of_pi}
f(X)\equiv \pi X+X^q\text{ mod }\pi X^2.
\end{equation}
For any $f$ satisfying \eqref{eq:LT_action_of_pi}, the Lubin-Tate formal module $(F,\alpha)$ such that $[\pi]_F=f$ is described as follows:

\begin{itemize}
\item the power series $F(X,Y)$ is the unique power series such that
$$F(X,Y)\equiv X+Y\text{ mod }(X,Y)^2,\quad F(f(X),f(Y))=f(F(X),F(Y));$$

\item for $r\in\cO_K$ the power series $[r]_F(X)$ is the unique power series such that
$$[r]_F(X)\equiv rX\text{ mod }X^2,\quad [r]_F(f(X))=f([r]_F(X)).$$
\end{itemize}

Choose a Lubin-Tate formal $\cO_K$-module $(F,\alpha),$ and consider its $\m_{K^{\sep}}$-valued points. We obtain a non-standard structure of an abelian group on the set $\m_{K^{\sep}},$ and the structure of an $\cO_K$-module on this abelian group:
$$x\underset{F}{+}y=F(x,y),\quad r\underset{F}{\cdot}x=[r]_F(x),\quad x,y\in\m_{K^{\sep}},\,r\in\cO_K.$$

Denote by $\m_{K^{\sep}}[\pi^{\infty}]$ the $\cO_K$-module of $\pi$-power torsion elements of $\m_{K^{\sep}}$ (with respect to the above $\cO_K$-module structure). Then we have a non-canonical $\cO_K$-linear isomorphism $\m_{K^{\sep}}[\pi^{\infty}]\cong K/\cO_K.$ It follows that we have a canonical isomorphism of rings
$$\cO_K\cong\End_{\cO_K}(\m_{K^{\sep}}[\pi^{\infty}]).$$

\begin{theo}\label{th:LT_main}\cite{LT65} 1) The extension of $K$ generated by the subset $\m_{K^{\sep}}[\pi^{\infty}]\subset K^{\sep}$ is exactly the field $K_{\pi}.$

2) The canonical map $\psi_{K,\pi}:\Gal(K_{\pi}/K)\to \Aut_{\cO_K}(\m_{K^{\sep}}[\pi^{\infty}])\cong \cO_K^*$ is an isomorphism.

3) The composition $\psi_{K,\pi}\circ\bbar{\theta}_{K,\pi}:\cO_K^*\to\cO_K^*$ sends $u$ to $u^{-1}.$ 
\end{theo}

We choose the Lubin-Tate formal module $(F,\alpha)$ for which we have $f(X)=[\pi]_F(X)=X^q+\pi X.$ Denote by $f^{[n]}(X)$ the $n$-th iteration of $f,$ where $f^{[0]}(X)=X.$ Then for $n>0$ the primitive $\pi^n$-torsion elements of $\m_{K^{\sep}}[\pi^{\infty}]$ are exactly the roots of the (separable) Eisenstein polynomial \begin{equation}\label{eq:formula_for_g_n} g_n(X)=\frac{f^{[n]}(X)}{f^{[n-1]}(X)}=(f^{[n-1]}(X))^{q-1}+\pi\end{equation} over $K.$ These roots have valuation $\frac{v(\pi)}{q^n-q^{n-1}}.$ The following statement is immediate from Theorem \ref{th:LT_main}. 

\begin{cor}\label{cor:LT_for_comparing_uniformizers} Let $\pi$ and $\varpi$ be uniformizers of a local field $K$ with the residue field $\FF_q,$ let $n$ be a positive integer, and consider the polynomial $g_n(X)$ defined in \eqref{eq:formula_for_g_n} (corresponding to the uniformizer $\pi$). The following are equivalent.

\begin{enumerate}[label=(\roman*),ref=(\roman*)]
\item We have $\varpi\equiv \pi\text{ mod }\pi^{n+1}.$

\item There exists a root $y$ of $g_n$ such that $\theta_K(\varpi)(y)=y.$
\end{enumerate}\end{cor} 

We now use this to obtain a useful tool for approximating the uniformizer $\can_K(\pi),$ introduced in Theorem \ref{th:main}. Since we want this uniformizer to be $-\pi,$ we replace $X^q+\pi X$ by $X^q-\pi X.$ 

\begin{lemma}\label{lem:LT_for_approximating_can_K} Let $\pi$ be a uniformizer in a local field $K,$ and we identify $\cO_K/\pi\cong\FF_q.$ Let $n$ be a positive integer. Denote by $f^{[n]}$ the $n$-th iteration of the polynomial $f(X)=X^q-\pi X,$ and put $g_n(X)=(f^{[n-1]}(X))^{q-1}-\pi.$   
Then the following are equivalent.
\begin{enumerate}[label=(\roman*),ref=(\roman*)]
\item We have $\can_K(\pi)\equiv -\pi\text{ mod }\pi^{n+1}.$ \label{can_K_close}

\item There exists a root of $g_n$ in $HW_{(K,\pi)}(\FF_q).$ \label{root}

\item There exists an element $y\in HW_{(K,\pi)}(\FF_q)$ such that $v(g_n(y))>n v(\pi).$ \label{almost_root} 
\end{enumerate}
\end{lemma}

\begin{proof}
The equivalence \Iff{can_K_close}{root} is a direct application of Corollary \ref{cor:LT_for_comparing_uniformizers} to the uniformizers $-\pi$ and $\can_K(\pi):$ a root $y$ of $g_n$ is contained in $HW_{(K,\pi)}(\FF_q)$ if and only if $$y=\varphi(y)=\theta_K(\can_K(\pi))(y).$$ Tautologically, we have \Implies{root}{almost_root}.

It remains to prove \Implies{almost_root}{root}. We normalize the valuation on $HW_{(K,\pi)}(\bbar{\FF}_q)$ so that $v(\pi)=1.$ Let $y$ be as in \ref{almost_root}, and denote by $\beta_1,\dots,\beta_{q^n-q^{n-1}}$ the roots of $g_n(X)$ in $HW_{(K,\pi)}(\bbar{\FF}_q).$ By Krasner's lemma \cite{O17, Kra46}, it suffices to prove that there exists $i\in\{1,\dots,q^n-q^{n-1}\}$ such that $v(y-\beta_i)>v(y-\beta_j)$ for $j\neq i.$ We will use the following obvious statement on the valuations of the differences of the roots of $g_n.$ If $i\ne j,$ then $v(\beta_i-\beta_j)=\frac{1}{q^m-q^{m-1}}$ if the difference of $\beta_i$ and $\beta_j$ in $\m_{K^{\sep}}[\pi^{\infty}]$ is a primitive $\pi^m$-torsion element.  

Now the assumption $v(g_n(y))>n$ means that 
$$\sum\limits_{i=1}^{q^n-q^{n-1}}v(y-\beta_i)>n\geq 1.$$ We may and will assume that $v(y-\beta_1)>\frac{1}{q^n-q^{n-1}}.$ There are exactly $q^{n-1}$ roots $\beta_i$ such that $v(\beta_1-\beta_i)>\frac{1}{q^n-q^{n-1}}.$ We will assume that these are $\beta_1,\dots,\beta_{q^{n-1}}.$ If $n=1,$ then we are done. Assuming $n\geq 2,$ we see that
$$\sum\limits_{i=1}^{q^{n-1}}v(y-\beta_i)>n-\sum\limits_{i=q^{n-1}+1}^{q^n-q^{n-1}}v(y-\beta_i)=n-\frac{q^n-2q^{n-1}}{q^n-q^{n-1}}=n-1+\frac{1}{q-1}\geq \frac{q}{q-1}.$$ 
We may and will assume that $v(y-\beta_1)>\frac{1}{q^{n-1}}\cdot\frac{q}{q-1}=\frac{1}{q^{n-1}-q^{n-2}}.$ There are exactly $q^{n-2}$ roots $\beta_i$ such that $v(\beta_1-\beta_i)>\frac{1}{q^{n-1}-q^{n-2}},$ and we will assume that these are $\beta_1,\dots,\beta_{q^{n-2}}.$ If $n=2,$ then we are done. Assuming $n\geq 3,$ we obtain
\begin{multline*}\sum\limits_{i=1}^{q^{n-2}}v(y-\beta_i)>n-1+\frac{1}{q-1}-\sum\limits_{i=q^{n-2}+1}^{q^{n-1}}v(y-\beta_i)=n-1+\frac{1}{q-1}-\frac{q^{n-1}-q^{n-2}}{q^{n-1}-q^{n-2}}\\
=n-2+\frac{1}{q-1}\geq \frac{q}{q-1}.\end{multline*}
Continuing the process, we eventually find a root, say $\beta_1,$ such that $v(y-\beta_1)>\frac{1}{q-1}$ and $v(y-\beta_i)\leq \frac{1}{q-1}$ for $i\ne 1.$ By Krasner's lemma, we have $\beta_1\in HW_{(K,\pi)}(\FF_q),$ which proves the implication \Implies{almost_root}{root}.
\end{proof}

\section{Proof of Theorem \ref{th:main}}
\label{sec:proof_of_main}

 Let $K$ be a local field, and let $\pi\in K$ be a uniformizer, and we identify $\cO_K/\pi\cong \FF_q.$ As in Lemma \ref{lem:LT_for_approximating_can_K}, consider the polynomial $f(X)=X^q-\pi X,$ denote by $f^{[n]}$ its iterations, and put $$g_n(X)=\frac{f^{[n]}(X)}{f^{[n-1]}(X)} =(f^{[n-1]}(X))^{q-1}-\pi.$$

The general idea of the proof of Theorem \ref{th:main} is the following. We will first apply Lemma \ref{lem:LT_for_approximating_can_K}, \Implies{root}{can_K_close}, to prove Theorem \ref{th:main} in the equal characteristic case. Then we will apply Lemma \ref{lem:LT_for_approximating_can_K}, \Implies{almost_root}{can_K_close}, to prove that $\can_K(\pi)$ and $-\pi$ are ``very close'' when $K$ is sufficiently ramified over $\Q_p$ (the precise statement is Lemma \ref{lem:main_th_for_deeply_ramified}). Then we will use Proposition \ref{prop:can_related_via_Norm} and Lemma \ref{lem:norms_for_wildly_ramified} to prove Theorem \ref{th:main} in the mixed characteristic case.

\subsection{The equal characteristic case.}
\label{ssec:equal_char}

Suppose that $K$ has characteristic $p,$ i.e. $K\cong \FF_q((\pi)).$ In this case we have $HW_{(K,\pi)}(\FF_q)\cong \FF_q((\pi^{\Q})).$ 
We directly write down a root of $g_n(X)$ in $\FF_q((\pi^{\Q}))$ for each $n>0.$ We put
\begin{equation}\label{eq:formula_for_y_n} y_n:=\sum\limits_{0<k_1<\dots<k_{n-1}}\pi^{\frac{1}{q-1}-q^{-k_1}-\dots-q^{-k_{n-1}}},\end{equation}
where the summation is over all $(n-1)$-tuples of strictly increasing positive integers (for $n=1$ we have $y_1=\pi^{\frac{1}{q-1}}$). Note that the lowest degree term of $y_n$ is given by $$\pi^{\frac{1}{q-1}-q^{-1}-\dots-q^{-n+1}}=\pi^{\frac{1}{q^n-q^{n-1}}}.$$ We have $f(y_1)=0,$ and for $n>1$ we have
\begin{multline*}f(y_n)= \sum\limits_{0<k_1<\dots<k_{n-1}} \pi^{\frac{q}{q-1}-q^{1-k_1}-\dots-q^{1-k_{n-1}}} - \sum\limits_{0<k_1<\dots<k_{n-1}}\pi^{\frac{q}{q-1}-q^{-k_1}-\dots-q^{-k_{n-1}}} =\\
 \sum\limits_{0\leq k_1<\dots<k_{n-1}} \pi^{\frac{q}{q-1}-q^{-k_1}-\dots-q^{-k_{n-1}}} - \sum\limits_{0<k_1<\dots<k_{n-1}}\pi^{\frac{q}{q-1}-q^{-k_1}-\dots-q^{-k_{n-1}}}=\\
 \sum\limits_{0 = k_1<\dots<k_{n-1}} \pi^{\frac{1}{q-1}-q^{-k_2}-\dots-q^{-k_{n-1}}} = y_{n-1}.
 \end{multline*} 

It follows that $y_n$ is a root of $f^{[n]}(X),$ but not a root of $f^{[n-1]}(X).$ Hence, $y_n$ is a root of $g_n(X).$ By Lemma \ref{lem:LT_for_approximating_can_K} we conclude that $\can_K(\pi)\equiv -\pi\text{ mod }\pi^{n+1}$ for all $n>0,$ hence $\can_K(\pi)=-\pi.$ This proves Theorem \ref{th:main} in the equal characteristic case.

\subsection{The mixed characteristic case}
\label{ssec:mixed_char}

Now we consider the mixed characteristic case, i.e. suppose that $K$ is a finite extension of $\Q_p.$ We normalize all the valuations so that $v(p)=1.$ Given the integers $n>0,$ $k\geq 0,$ consider the following statement:
$$S(n,k):\quad \text{if }v(\pi)\leq \frac{1}{p^k} \text{ (i.e. }e(K/\Q_p)\geq p^k\text{), then }\can_K(\pi)\equiv -\pi\text{ mod }\pi^{n+1}.$$
We first prove it for $n\leq p^k.$

\begin{lemma}\label{lem:main_th_for_deeply_ramified} For $m=e(K/\Q_p),$ we have $\can_K(\pi)\equiv -\pi\text{ mod }\pi^{m+1}.$ In particular, the statement $S(n,k)$ holds for $n\leq p^k.$ 
\end{lemma}

\begin{proof}
Denote by $z_m\in HW_{(K,\pi)}(\FF_q)$ the element given by the same expression as $y_m$ from \eqref{eq:formula_for_y_n}, i.e.
$$z_m=\sum\limits_{0<k_1<\dots<k_{m-1}}\pi^{\frac{1}{q-1}-q^{-k_1}-\dots-q^{-k_{m-1}}}.$$

By Lemma \ref{lem:LT_for_approximating_can_K}, it suffices to prove that $v(g_m(z_m))>m v(\pi)=1.$ We claim that this formally follows from the above computation in characteristic $p.$ To make this precise, recall that $HW_{(K,\pi)}(\FF_q)=\cO_K((t^{\Q}))/(t-\pi).$ The element $z_m\in HW_{(K,\pi)}(\FF_q)$ is the image of the element $$\tilde{z}_m=\sum\limits_{0<k_1<\dots<k_{n-1}} t^{\frac{1}{q-1}-q^{-k_1}-\dots-q^{-k_{m-1}}}\in\Z((t^{\Q})).$$ Consider the polynomial $\tilde{f}(X)=X^q-tX\in\Z[t][X].$ We define $\tilde{f}^{[m]}(X),\tilde{g}_m(X)\in\Z[t][X]$ similarly. 
The above computation in the equal characteristic case implies that $\tilde{g}_m(\tilde{z}_m)$ is divisible by $p$ in $\Z((t^{\Q})).$ Moreover, since the constant term of the polynomial $\tilde{g}_m(X)$ equals $-t$ and since $\tilde{z}_m$ is a sum of strictly positive powers of $t,$ we conclude that all the exponents in the Hahn series expression of $\tilde{g}_m(\tilde{z}_m)$ are $\geq\veps$ for some rational $\veps>0.$ Applying the natural homomorphisms $\Z[t]\to\cO_K$ and $\Z((t^{\Q}))\to HW_{(K,\pi)}(\FF_q),$ sending $t$ to $\pi,$ we conclude that $g_m(z_m)$ is divisible by $p \pi^{\veps}$ in the valuation ring of $HW_{(K,\pi)}(\FF_q).$ Hence, $$v(g_m(z_m))\geq v(p\pi^{\veps})> 1,$$
as required.
\end{proof}

Next, we slightly improve the bound from Lemma \ref{lem:main_th_for_deeply_ramified}.

\begin{lemma}\label{lem:improved_bound} Suppose that $v(\pi)\leq \frac{p}{(p-1)n}$ for some integer $n>0.$ Then we have $\can_K(\pi)\equiv -\pi\text{ mod }\pi^{n+1}.$ In particular, the statement $S(n,k)$ holds for $n\leq \frac{p^{k+1}}{p-1}.$
\end{lemma} 

\begin{proof} Consider the extension $K'=K(\pi^{\frac1p})$ of $K.$ We have $v(\pi^{\frac1p})\leq \frac{1}{(p-1)n}\leq \frac1n.$ Hence, by Lemma \ref{lem:main_th_for_deeply_ramified} we have $\can_{K'}(\pi^{\frac1p})\equiv -\pi^{\frac1p}\text{ mod }\pi^{\frac{n+1}{p}}.$ 

Now, by Proposition \ref{prop:can_related_via_Norm}, we have $\Nm_{K'/K}(\can_{K'}(\pi^{\frac1p}))=\can_K(\pi).$ Note that $\Nm_{K'/K}(-\pi^{\frac1p})=-\pi.$ Our assumption on $n$ means exactly that $(n-\frac{n}{p})v(\pi)\leq 1=v(p).$ Applying Lemma \ref{lem:norms_for_wildly_ramified}, we obtain
$$\frac{\can_K(\pi)}{-\pi}=\Nm_{K'/K}\left(\frac{\can_{K'}(\pi^{\frac1p})}{-\pi^{\frac{1}{p}}}\right)\in \Nm_{K'/K}(1+\pi^{\frac{n}{p}}\cO_{K'})\subset 1+\pi^n\cO_K,$$
as required.
\end{proof}

Similarly, we obtain the following reduction.

\begin{lemma}\label{lem:reduction_for_S_n_k} Let $n>0$ and $k\geq 0$ be such that $n>\frac{p^{k+1}}{p-1}.$ Then the statement $S(pn-p^{k+1},k+1)$ implies the statement $S(n,k).$\end{lemma}

\begin{proof}The argument is basically the same as for Lemma \ref{lem:improved_bound}. Suppose that $v(\pi)\leq \frac{1}{p^k}$ and suppose that the statement $S(pn-p^{k+1},k+1)$ holds. Consider the extension $K'=K(\pi^{\frac1p})$ of $K.$ Then $v(\pi^{\frac1p})\leq \frac{1}{p^{k+1}},$ and applying $S(pn-p^{k+1},k+1)$ we obtain
$$\can_{K'}(\pi^{\frac1p})\equiv -\pi^{\frac1p}\text{ mod }\pi^{\frac1p+n-p^k}.$$ Our assumptions imply that $p(n-p^k)\geq n$ and $$(n-(n-p^k))v(\pi)=p^k v(\pi)\leq 1=v(p).$$
Applying Proposition \ref{prop:can_related_via_Norm} and Lemma \ref{lem:norms_for_wildly_ramified}, we obtain
$$\frac{\can_K(\pi)}{-\pi}=\Nm_{K'/K}\left(\frac{\can_{K'}(\pi^{\frac1p})}{-\pi^{\frac{1}{p}}}\right)\in \Nm_{K'/K}(1+\pi^{n-p^k}\cO_{K'})\subset 1+\pi^n\cO_K,$$
as required.\end{proof}

Now we prove all the statements $S(n,k)$ by induction on the positive integer parameter $r=\lceil \frac{n}{p^k}\rceil.$ The base of induction $r=1$ follows from Lemma \ref{lem:main_th_for_deeply_ramified}.

Let $r>1$ and suppose that the statements $S(n,k)$ hold for $\lceil\frac{n}{p^k}\rceil\leq r-1.$ Take some $n,k$ such that $\lceil\frac{n}{p^k}\rceil=r.$ If $n\leq \frac{p^{k+1}}{p-1}$ (which is possible only for $r=2$), then $S(n,k)$ holds by Lemma \ref{lem:improved_bound}. Assuming $n>\frac{p^{k+1}}{p-1},$ it follows from the induction hypothesis that the statement $S(pn-p^{k+1},k+1)$ holds: we have
$$\left\lceil \frac{pn-p^{k+1}}{p^{k+1}} \right\rceil = \left\lceil \frac{n}{p^k}-1 \right\rceil = r-1.$$
Applying Lemma \ref{lem:reduction_for_S_n_k}, we conclude that $S(n,k)$ holds, as required.

So we proved the statements $S(n,k)$ for all $n>0,$ $k\geq 0.$ Therefore, without any assumption on the ramification index $e(K/\Q_p),$ we have $\can_K(\pi)\equiv -\pi\text{ mod }\pi^{n+1}$ for all $n>0,$ i.e. $\can_K(\pi)=-\pi.$ This proves Theorem \ref{th:main} in the mixed characteristic case.



\begin{thebibliography}{Hahn95}

\bibitem[Bo11]{Bo11}J.~Borger, ``The basic geometry of Witt vectors, I: The affine case''. Algebra and Number Theory 5 (2011), no. 2, pp 231-285.

\bibitem[CF67]{CF67}J.W.S.~Cassels, A.~Fr\"ohlich (Eds.),  ``Algebraic Number Theory''. Academic Press, 1967.

\bibitem[Dr76]{Dr76} V.G.~Drinfeld, ``Coverings of $p$-adic symmetric domains''. Functional Analysis and its Applications 10 (1976), no. 2, p. 29-40.

\bibitem[Hahn95]{Hahn95}H.~Hahn, ``\"Uber die nichtarchimedische Gr\"o{\ss}ensysteme''. In Gesammelte Abhandlungen I,
Springer-Verlag, 1995.

\bibitem[Haz78]{Haz78} M.~Hazewinkel, ``Formal groups and applications''. Pure and Applied
Mathematics, vol. 78, Academic Press, 1978.

\bibitem[J85]{J85} A.~Joyal, ``$\delta$-anneaux et vecteurs de Witt''. C. R. Math. Rep. Acad. Sci. Canada, 7(3):177-182, 1985.

\bibitem[Ke01]{Ke01} K.S.~Kedlaya, ``Power series and $p$-adic algebraic closures''. Journal of Number Theory, Volume 89, Issue 2, 2001, pp. 324-339.

\bibitem[Kra46]{Kra46} M.~Krasner, ``Th\'eorie non ab\'elienne des corps de classes pour les extensions finies et s\'eparables des corps valu\'es complets: principes fondamentaux; espaces de polynomes et transformation T ; lois d'unicit\'e, d'ordination et d'existence''.
C. R. Acad. Sci. Paris 222 (1946), 626-628.

\bibitem[Kru32]{Kr32} W.~Krull, ``Allgemeine Bewertungstheorie''. J. f\"ur Math. 167 (1932), 160-196.

\bibitem[La86]{La86} D.~Lampert, ``Algebraic $p$-adic expansions''. J. Number Theory 23 (1986) 279-284.

\bibitem[LT65]{LT65} J.~Lubin, J. Tate, ``Formal complex multiplication in local fields''. Ann. Math.
81 (1965), 380-387.

\bibitem[M20]{M20} J.S.~Milne, ``Class Field Theory''. Course Notes (v4.03), available at: https://www.jmilne.org/math/CourseNotes/CFT.pdf (2020).

\bibitem[N99]{N99} J.~Neukirch, ``Algebraic Number Theory''. Springer-Verlag, Berlin, (1999).

\bibitem[O17]{O17} A.~Ostrowski, ``\"Uber sogenannte perfekte K\"orper''. J. Reine Angew. Math. 147
(1917), 191-204.

\bibitem[Pa77]{Pa77} D.~Passman, ``The Algebraic Structure of Group Rings''. John Wiley and
Sons, 1977, 598-602.

\bibitem[Po93]{P93} B.~Poonen, ``Maximally complete fields''. Enseign. Math. 39 (1993) 87-106.

\end{thebibliography}
\end{document}